\documentclass[11pt]{amsart}
\usepackage{amsfonts}
\usepackage{amsmath}
\usepackage{color}
\usepackage[mathscr]{eucal}
\usepackage{amscd, latexsym, graphicx, mathrsfs, enumerate}
\usepackage{amssymb}

\makeindex

\newtheorem{thm}{{{Theorem}}}[section]

\newtheorem{lem}[thm]{{Lemma}}

\newtheorem{remark}[thm]{Remark}
\numberwithin{equation}{section}
\newtheorem{Def}[equation]{Definition}

\def\Z{\mathbb{Z}}
\def\Q{\mathbb{Q}}

\def\C{\mathbb{C}}
\def\A{\mathbb{A}}

\def\O{{\mathop{\mathrm{O}}}}

\def\diag{{\mathop{\mathrm{diag}}}}

\def\bs{{\backslash}}
\def\ds{\displaystyle}

\def\lra{{\longrightarrow}}

\def\tf{{\tilde{f}}}

\newcommand{\oF}{\overline{F}}
\newcommand{\lla}{\longleftarrow}

\renewcommand{\O}{\mathcal{O}}

\newcommand{\tG}{\widetilde{G}}
\newcommand{\wf}{\widetilde{f}}

\numberwithin{equation}{section}

\title[Transfers of some Hecke elements for possibly ramified base change  in $GL_n$]
{Transfers of some Hecke elements for possibly ramified base change in $GL_n$}
\author{Takuya Yamauchi}
\date{\today}

\thanks{The author is partially supported by 
JSPS KAKENHI Grant Number (B) No.19H01778.}

\address{Takuya Yamauchi \\
Mathematical Inst. Tohoku Univ.\\
 6-3,Aoba, Aramaki, Aoba-Ku, Sendai 980-8578, JAPAN}
\email{takuya.yamauchi.c3@tohoku.ac.jp}

\begin{document}
\maketitle

\begin{abstract}
In this paper, we prove an explicit matching theorem for 
some Hecke elements in the case of (possibly ramified) cyclic base 
change for general linear groups over local fields of characteristic zero with 
odd residual characteristic under a mild assumption. 
A key observation, based on the works of  Waldspurger and Ganapathy-Varma, is to regard the base change lifts with twisted endoscopic lifts and 
replace the condition for the matching orbital integrals with one for semi-simple descent 
in the twisted space according to Waldspurger's fundamental work ``L'endoscopie tordue n'est pas si tordue''. 
\end{abstract}

\tableofcontents

\section{Introduction}\label{ntro}
The transfers of Hecke elements play a fundamental role to study  
automorphic representations in Langlands correspondences. 
Thanks to many people including Waldspurger, Ngo among others, 
the existence of a transfer for any Hecke element is now guaranteed 
(cf. \cite{NgoICM}). 
However, to describe the transfer explicitly is a different matter. 
In this article we prove an explicit transfer theorem for the characteristic 
function of any principal congruence subgroup for any cyclic base change in the case of 
the general linear group $GL_n$ over any local field of odd residual characteristic with a mild condition. 
The unramified case in the unit elements of 
Hecke algebras associated to hyperspecial open compact subgroups is already done by 
Kottwitz (see p.239 of \cite{Ko}) in more general setting and therefore, the ramified case has been remained to be open though some experts might have known.   

Let us set up some notation to explain our results. 
Let $p$ be an odd prime number and $F$ a finite extension of $\Q_p$. 
Let $E$ be a cyclic extension of $F$ with the generator $\theta$ of the 
Galois group ${\rm Gal}(E/F)$. Put $d=[E:F]$. Let $G={\rm Res}_{E/F}(GL_n/E)$ 
be the Weil restriction of the general linear group $GL_n/E$ to $F$ and let $H=GL_n/F$. 
The action of $\theta$ is naturally extended to $G(F)=GL_n(E)$ and we also denote it by 
$\theta$ again. 
Let $\O_E$ (resp. $\O_F$) be the ring of integers of $E$ (resp. $F$) with 
a uniformizer $\varpi_E$ (resp. $\varpi_F$). We denote by $e=e(E/F)$  
the ramification index so that $\varpi^e_E$ and $\varpi_F$ differ by a unit in $\O_E$. 
For a positive integer $m$, put $K_E(m):=I_n+\varpi^m_E M_n(\O_E)$ where 
$I_n$ stands for the identity matrix of size $n$. Similarly, 
put $K_F(m):=I_n+\varpi^m_F M_n(\O_F)$. 
For each open compact subgroup $C$ of $G(F)$ or $H(F)$, 
we denote by $\textbf{1}_C$ the characteristic function of $C$. 
Let us fix Haar measures 
$dg$ on $G(F)$ and $dh$ on $H(F)$ respectively. 
Then we will prove the following: 
\begin{thm}\label{main1}Keep the notation as above. Let $m$ be a positive integer. 
Suppose that the residual characteristic $p$ of $F$ is odd and prime to $d=[E:F]$. 
The Hecke element ${\rm vol}^{-1}(K_E(m)_\theta)\textbf{1}_{K_E(m)_\theta}$  in $C^\infty_c(GL_n(F))$ is a transfer $($in the sense of Definition \ref{moi}$)$ 
of  the Hecke element 
${\rm vol}^{-1}(K_E(m))\textbf{1}_{K_E(m)}$ in 
$C^\infty_c(G(F))=C^\infty_c(GL_n(E))$ for the base change lift 
in $(G,GL_n/F)$. Here  
${\rm vol}$ stands for the volume with respect to the fixed measure.   
\end{thm}
This will be applied to the endoscopic classification of Siegel cusp forms 
(see Section 5 of \cite{KWY}) to make a detour considering quasi-split but not split 
orthogonal groups in Arthur's classification. 
\begin{remark}
For a positive integer $m$, we write $m=ke+r$ with unique integers $k\ge 0$ and 
$0\le r<e$. 
Then we see easily that 
$$K_E(m)_\theta=
\left\{\begin{array}{cl}
K_F(k+1) & (\text{if $r\neq 0$}) \\
K_F(k) &  (\text{if $r=0$})
\end{array}\right..$$ 
 
\end{remark}
The following theorem is a more finer version of matching orbital integrals which is an extension of 
Proposition 3.1, p. 20 of \cite{AC} for the Hecke elements in Theorem \ref{main1}. 
\begin{thm}\label{main2} Keep the notation and the assumption on $F$ and $E/F$ 
in Theorem \ref{main1}. 
Let $m$ be a positive integer. 
Suppose that the residual characteristic $p$ of $F$ is odd and prime to $d=[E:F]$. 
Put $f^H:=
{\rm vol}^{-1}(K_E(m)_\theta)\textbf{1}_{K_E(m)_\theta}$  in $C^\infty_c(H(F))=
C^\infty_c(GL_n(F))$ and $f^G:={\rm vol}^{-1}(K_E(m))\textbf{1}_{K_E(m)}$ in 
$C^\infty_c(G(F))=C^\infty_c(GL_n(E))$. 
Put $\wf^G={\rm vol}^{-1}(K_E(m))\textbf{1}_{K_E(m)\rtimes\theta}$ in 
$C^\infty_c(\tG(F))$ which corresponds to $\tf^G$ under $G(F)\stackrel{\sim}{\lra}
\tG(F),\ g\mapsto g\rtimes \theta ($see Section \ref{tc} for the twisted space $\tG)$. 
For a semisimple element $\gamma\in H(F)$ and 
a semisimple element $\delta\in G(F)$,  
we have the matching of normalized orbital integrals
$$
I_{\gamma}(f^H,dh,dh_\gamma)=
\left\{\begin{array}{cl} 
I_{\delta\rtimes \theta}(\wf^G,dg,dg_{\delta\rtimes\theta}) 
& \text{if $\gamma$ is a norm of $\delta$} \\
0  & \text{if $\gamma$ is not a norm}
\end{array}\right..
$$
See $($\ref{noiG1}$)$, $($\ref{noiG2}$)$, and $($\ref{noiH}$)$ for 
the normalized orbital integrals. 
\end{thm}

\begin{remark}\label{on-assumption}
Our assumption on the residual characteristic $p$ is crucial in the following points:
\begin{enumerate}
\item being odd is necessary to guarantee (\ref{homeo});
\item being prime to $d=[E:F]$ is necessary in the computation 
done in Section \ref{tc} through Section \ref{sdt}  around 
topologically unipotent elements. In particular Theorem \ref{mainSSD} is 
important among others. 
\end{enumerate}
\end{remark}

\begin{remark}\label{current-situation-for-BC}
\begin{enumerate}
\item Thanks to works \cite{HI}, \cite{CR} with the results in this paper, 
an explicit transfer theorem for any base change lifts and 
automorphic inductions has been developed  
in great general so that we can apply these results to several arithmetic problems.  
\item Explicit transfer theorems for various settings of twisted endoscopy data, 
based on Lie-algebra calculation, are now available due to important works \cite{Fe}, \cite{GV}, and \cite{Oi} which embody  Waldspurger's philosophy in \cite{Wal-book} saying 
``Twisted endoscopy is not so twisted''. 
\end{enumerate}
\end{remark}

We give a global application. Let $K$ be a number field and $L/K$ be a cyclic extension of 
 degree $d$. 
By abusing notation, let $\theta$ be a generator of ${\rm Gal}(L/K)$.  
Let $\pi=\pi_f\otimes \pi_\infty$ be a cuspidal automorphic representation of $GL_n(\A_K)$ 
where $\pi_f$ stands for the finite part of $\pi$. 
We denote by $\Pi:={\rm BC}(\pi)=\Pi_f\otimes \Pi_\infty$ 
the base change lift of $\pi$ to 
$GL_n(\A_L)$ (see Chapter 3, Section 6 of \cite{AC}).  
Let $S_K$ be the set of all finite places $v$ of $K$ such that $L/K$ is ramified at $v$. 
Let $S_L$ be the set of all finite places $w$ of $L$ lying over some place in $S_K$.  
Let $S$ be a finite set of finite places of $L$ including $S_L$. 
For each $\underline{n}=(n_w)_{w\in S\setminus S_L}\in \Z^{|S\setminus S_L|}_{\ge 1}$ 
and  each $\underline{m}=(m_w)_{w\in S_L}\in \Z^{|S_L|}_{\ge 1}$, 
put 
$$U_{S}(\underline{n},\underline{m}):=
\ds\prod_{w\in S\setminus S_L}K_{L_w}(n_w)\times 
\prod_{w\in S_L}K_{L_w}(m_w)\times 
\prod_{w\not\in S}GL_n(\mathcal{O}_{L_w})$$ 
which is an open compact subgroup of $GL_n(\A^{(\infty)}_L)$ 
where $\A^{(\infty)}_L$ stands for the ring of finite adeles of $L$. 
Let $U_{S}(\underline{n},\underline{m})_\theta$ be the $\theta$-fixed 
part of $U_{S}(\underline{n},\underline{m})$ which can be regarded as an 
open compact subgroup of $GL_n(\A^{(\infty)}_K)$.  
\begin{thm}\label{main3}
Assume $\Pi$ is cuspidal and $U_{S}(\underline{n},\underline{m})$ is stable under $\theta$. 
Assume that for each $v\in S_K$ the residual characteristic $p_v$ is odd and 
prime to $d_v=[L_w:K_v]$ for any finite place $w$ of $L$ lying over $v$. It holds that   
$${\rm dim}(\pi^{U_{S}(\underline{n},\underline{m})_\theta}_f)
\le  {\rm dim }(\Pi^{U_{S}(\underline{n},\underline{m})}_f).$$
\end{thm}

This paper is organized as follows. 
In Section \ref{tes}, according to the notation of \cite{KS}, we interpret base change lifts 
in terms of the twisted endoscopic lifts which has been itself well-known and not new at all. 
Since the author could not find any suitable reference, we state it more than necessary. 
In Section \ref{Norm} we recall the norm map (in Section 3 of \cite{KS}) for the base change lifts, equivalently, for  our 
twisted endoscopic data. The author would find a related explanation in Section 1 of \cite{AC}. However, we fit into the notation in \cite{KS} 
as much as possible. In Section \ref{tc} through Section \ref{oi} we study 
basic facts about  (normalized) orbital integrals and recall the definition of 
semi-simple descent. 
We largely follow the argument in Ganapathy-Varma's paper \cite{GV}. 
In Section \ref{tf}, we compute the transfer factor $\Delta^{{\rm IV}}$ with 
``upper IV''. 
The author really owe Professor Waldspurger a lot with email correspondences  in the computation of $\Delta^{{\rm IV}}$. 
A key in our matching theorem is the semi-simple descent theorem whose proof 
is devoted to Section \ref{sdt}. 
In Section \ref{proof}, we give proofs of main theorems.

\textbf{Acknowledgments.}
This work has been raised toward an application to 
equidistribution theorems for $Sp_{2n}$ in \cite{KWY}. 
I really appreciate my collaborators (H-H .Kim ans S. Wakatsuki) to 
encourage me to prove the transfer theorem here.  
The author would also like to thank 
professors M. Oi, J-L. Waldspurger, and S. Yamana for helpful discussions. 
Professor Oi explained me about the importance of Ganapathy-Varma's paper \cite{GV} 
and also how the work of Waldspurger in 
the book \cite{Wal-book} should be understood. 
The special thanks are given to Professor Waldspurger again for explaining the author 
about how to compute the transfer factors in our setting. 
Finally the author would like to express his sincere gratitude to the anonymous referees for their carefully reading of this paper and for giving the author a lot of constructive
and valuable comments. In particular, the statement of Theorem \ref{main3} and 
its proof in an earlier version were collected according to the referee's suggestion. 

\section{A twisted endoscopic subgroup of $G={\rm Res}_{E/F}(GL_n/E)$}\label{tes}
Keep the notation in the previous section. We follow 
the notation in Chapter 2 of \cite{KS}. For any algebraic group $\mathcal{G}$ over 
a field, we denote by $\mathcal{G}^0$ the connected component of $\mathcal{G}$. 
Recall the connected reductive group $G={\rm Res}_{E/F}(GL_n/E)$ over $F$ 
which is quasi-split over $F$ since ${\rm Res}_{E/F}(B/E)$ is 
a Borel subgroup defined over $F$ where $B$ stands for the upper Borel subgroup of 
$H=GL_n/F$ and $B/E$ does for its base change to $E$. Note that it is an outer form of $(GL_n/F)^d$. We fix the quasi split datum ${\rm spl}_G$ associated to 
${\rm Res}_{E/F}(B/E)$ and its maximal 
torus $T$. 
We fix an isomorphism $\iota:{\rm Gal}(E/F)\stackrel{\sim}{\lra} \Z/d\Z,\ \sigma\mapsto \iota(\sigma)$ and extend it to the Weil group $W_F$ of $F$ via 
the natural projection $W_F\lra W_F/W_E\simeq {\rm Gal}(E/F)$. 
We also denote it by $\iota$ again. 
The Langlands dual group of $G$ is given by 
$\widehat{G}=\ds\prod_{i\in \Z/d\Z}GL_n(\C)$ with the action of $W_F$:
\begin{equation}\label{lga}
\sigma(g)=(g_{i+\iota(\sigma)})_{i\in \Z/d\Z},\ \sigma\in W_F,\ g=(g_i)_{i\in \Z/d\Z}\in \widehat{G}.
\end{equation}
It makes up the $L$-group ${}^L G:=\widehat{G}\rtimes W_F$. 
Let $\theta$ be a generator of $ {\rm Gal}(E/F)$. 
For any $F$-algebra $R$, the action of $\theta$ on $E\otimes_FR$ is given by 
$\theta(e\otimes r)=\theta(e)\otimes r$. 
It naturally induces an automorphism over $F$ of $G$ which is denoted by $\theta$ again. 
We denote by $\widehat{\theta}$ the automorphism of 
$\widehat{G}$ induced from the action of $\theta$ on $\widehat{G}$. 
The automorphism $\widehat{\theta}$ is acting on $\widehat{G}$ by 
$\widehat{\theta}((g_i)_{i\in\Z/d\Z})=(g_{i+\iota(\theta)})_{i\in \Z/d\Z}$ as in (\ref{lga}). Put ${}^L\theta=\widehat{\theta}\rtimes {\rm id}_{W_F}$. 
We consider the endoscopic datum $(G,\theta,\textbf{1}_{{\rm triv}})$ where 
$\textbf{1}_{{\rm triv}}$ is the trivial class in the Galois cohomology 
$H^1(W_F,Z(\widehat{G}))\simeq H^1(W_E,\C^\times)=
{\rm Hom}_{{\rm ct}}(W_E,\C^\times)$. Here the subscript ``ct'' means 
continuous. The trivial class $\textbf{1}_{{\rm triv}}$ can be regarded as 
the trivial character $W_F\lra Z(\widehat{G})\subset \widehat{G}\subset {}^L\widehat{G}$ 
via the transfer map $W^{{\rm ab}}_F\lra W^{{\rm ab}}_E$. Further we extend it 
to $\mathcal{H}:=GL_n(\C)\rtimes W_F$ via the natural projection $\mathcal{H}\lra W_F$ 
where $W_F$ acts trivially on $GL_n(\C)$ in the definition of $\mathcal{H}$. 
We write $\textbf{1}_{{\rm triv}}:\mathcal{H}\lra {}^L\widehat{G}$ for 
the resulting trivial character.   

In view of the base change, the endoscopic datum $(H,\mathcal{H},s,\xi)$ for 
$(G,\theta,\textbf{1}_{{\rm triv}})$  consists of 
\begin{itemize}
\item $H=GL_n/F$ with the Langlands dual $\widehat{H}=GL_n(\C)$;
\item $\mathcal{H}=GL_n(\C)\rtimes W_F$;
\item $s=(1)_{i\in \Z/d\Z}\in \widehat{G}$ where $1$ stands for the unit element 
in $GL_n(\C)$;
\item $\xi:\mathcal{H}\lra {}^L G,\ h\rtimes w\mapsto \diag(h)\rtimes w$ where 
$\diag:GL_n(\C)\lra (GL_n(\C))^d$ is the diagonal embedding. 
\end{itemize}
Clearly, for each $h\rtimes w\in \mathcal{H}$, 
$${\rm Int}(s)\circ {}^L\theta\circ \xi(h\rtimes w)=\widehat{\theta}(\diag(h))\rtimes w
=\diag(h)\rtimes w=\xi(h\rtimes w)=\textbf{1}_{{\rm triv}}(h\rtimes w)\cdot\xi(h\rtimes w).$$
The computation $${\rm Cent}_{\widehat{\theta}}(s,\widehat{G}):=\{g\in\widehat{G}\ |\  
gs\widehat{\theta}(g)^{-1}=s  \}=
\{g\in\widehat{G}\ |\  
\widehat{\theta}(g)=g  \}=\{\diag(h)\ |\ h\in GL_n(\C)\}$$
shows $\xi|_{\widehat{H}}:\widehat{H}\stackrel{\sim}{\lra}{\rm Cent}_{\widehat{\theta}}(s,\widehat{G})$. Therefore, (2.1.4a) and (2.1.4b) in p.17, Chapter 2 of \cite{KS} are satisfied. 
Further, 
$\xi(Z(\widehat{H})^{W_F})^0=\{\diag(a\cdot 1)\ |\ a\in \C^\times\}$ 
is a subset of $Z(\widehat{G})=\{(a_1\cdot 1,\ldots,a_d\cdot 1)\in \widehat{G}\ |\ a_1,\ldots,a_d\in \C^\times\}$. In conclusion, $(H,\mathcal{H},s,\xi)$ is an elliptic, 
twisted endoscopic datum for $(G,\theta,\textbf{1}_{{\rm triv}})$ and in fact, 
it is easy to check that any elliptic, 
twisted endoscopic datum for $(G,\theta,\textbf{1}_{{\rm triv}})$ is 
isomorphic to our datum. Throughout this paper, we consider only this 
endoscopic datum $(H,\mathcal{H},s,\xi)$. 

In view of a global application, 
being elliptic is related to the discrete spectrum of the $L^2$-space of 
automorphic forms in \cite{AC}. 

\section{The norm map} \label{Norm}
Let us keep the notation in the previous section. 
In this section we recall the norm map for our endoscopic datum $(H,\mathcal{H},s,\xi)$ for 
$(G,\theta,\textbf{1}_{{\rm triv}})$. We follow the notation in Chapter 3 of \cite{KS}. 
Let $Cl_{{\rm ss}}(H)$ be the set of semisimple conjugacy classes in 
$H(\oF)=GL_n(\oF)$ and write its element by 
$[\gamma],\ \gamma\in H(\oF)$.  
Two elements $g_1,g_2\in G(F)=GL_n(E)$ are said to be $\theta$-conjugate if 
there exists an element $g\in G(F)$ such that $g_1=g^{-1}g_2\theta(g)$.  
Let $Cl_{{\rm ss}}(G,\theta)$ be 
the set of $\theta$-semisimple $\theta$-conjugacy classes in $G(\oF)=
GL_n(E\otimes_F\oF)$ (see Definition \ref{image}-(2)). 
We shall describe the canonical map 
$\mathcal A_{H/G}:Cl_{{\rm ss}}(H)\lra Cl_{{\rm ss}}(G,\theta)$ 
in Theorem 3.3.A, p.27 of \cite{KS}. 
Let $(B_H,T_H)$ be a pair of the upper Borel subgroup of $H$ and the diagonal maximal 
torus $T_H$ in $B_H$. Let $\Omega_H=N_{H}(T_H)/T_H$ 
be the Weyl group of $H$ with respect to $(B_H,T_H)$ and it acts on 
$T\simeq (GL_1/F)^n$ as the permutation of the entries via the natural identification 
$\Omega_H\simeq \mathfrak{S}_n$.  
Pick a class $C\in Cl_{{\rm ss}}(H)$. By taking the Jordan normal form, 
we see that $C=[t_C]$ for some $t_C\in T_H(\oF)$. 
Therefore, we have a bijection 
\begin{equation}\label{h-conj}
Cl_{{\rm ss}}(H)\lra T_H(\oF)/\Omega_H,\ C\mapsto \widetilde{t}_C 
\end{equation}
where $\widetilde{t}_C$ is the class of $t_C$ in $T_H(\oF)/\Omega_H$. 
Let $T={\rm Res}_{E/F}(T_H/E)$ be the maximal torus in 
the Borel subgroup $B={\rm Res}_{E/F}(B_H/E)$ in $G$. Recall that 
$\theta$ acting on $G$ is quasi-semisimple, hence,  
$T$ and $B$ are stable under $\theta$. 
We define the map $1-\theta:T\lra T,\ t\mapsto t\cdot \theta(t)^{-1}$. 
Then we have an isomorphism 
\begin{equation}\label{tt}
T_\theta(\oF):=T(\oF)/(1-\theta)T(\oF)\lra T_H(\oF),\ (t_1,\ldots,t_d)\mapsto 
t_1\cdots t_d.
\end{equation}
where we use a fixed identification $T(\oF)\simeq (T_H(\oF))^d$ 
so that 
$\theta$ acts on $(T_H(\oF))^d$ by the shift operator $(t_1,t_2,\ldots,t_d)\mapsto (t_2,\ldots,t_d,t_1)$.  
Let $\Omega=N_G(T)/T$ be the Weyl group of $G$ 
with respect to $(B,T)$. As mentioned before, $\theta$ acts on $T$. 
If we fix an identification $\Omega\simeq \mathfrak{S}^d_n$, then 
the $\theta$-fixed part $\Omega^\theta$ is identified with the diagonal part 
$\Delta \mathfrak{S}_n:=\{(\sigma,\ldots,\sigma)\ |\ \sigma\in \mathfrak S_n \}$ of 
$\mathfrak{S}^d_n$. 
The map (\ref{tt}) naturally yields an isomorphism 
\begin{equation}\label{nn}
\widetilde{N}_{E/F}:T_\theta(\oF)/\Omega^\theta\lra T_H(\oF)/\Omega_H.
\end{equation}
Pick an element $C$ in $Cl_{{\rm ss}}(G,\theta)$. 
By Lemma 3.2.A of \cite{KS}, the image of $C\cap T(\oF)$ in $T_\theta(\oF)$ is 
a single oribit of $\Omega^\theta$ and we write such an orbit as 
$\Omega^\theta \widetilde{s}_C,\  \widetilde{s}_C\in T_\theta(\oF)$ 
for $C$.  
Thus we have a bijection 
\begin{equation}\label{g-conj}
Cl_{{\rm ss}}(G,\theta)\lra T_\theta/\Omega^\theta,\ C\mapsto \widetilde{s}_C.
\end{equation}
Putting everything together, we have 
\begin{equation}\label{norm}
\mathcal{A}_{H/G}:Cl_{{\rm ss}}(H)\stackrel{(\ref{h-conj}) \atop\sim}{\lra} 
 T_H(\oF)/\Omega_H
\stackrel{\widetilde{N}^{-1}_{E/F}}{\lra} 
T_\theta(\oF)/\Omega^\theta
\stackrel{(\ref{g-conj}) \atop\sim}{\lla} 
Cl_{{\rm ss}}(G,\theta).
\end{equation} 
\begin{Def}\label{image}
\begin{enumerate}\item
For a semisimple element $\gamma\in H(F)=GL_n(F)$ we say $\gamma$ is a norm of $\delta\in G(F)=GL_n(E)$ 
if $\delta$ lies in $\mathcal A_{H/G}([\gamma])\cap G(F)$ where $[\gamma]$ 
stands for a class of $\gamma$ in $Cl_{{\rm ss}}(H)$.  
Otherwise, we say $\gamma$ is not a norm of $\delta$. 
\item We say $\delta\in G(F)$ is $\theta$-semisimple if there exists 
$g\in G(F)$ such that $g^{-1}\delta \theta(g)$ lies in $T(F)=T_H(E)$.  
This is equivalent to enjoying the condition explained in Section 1.3,  p.15 of \cite{KS}.  
\item We say a $\theta$-semisimple element is $\theta$-regular if 
$G_{\delta\rtimes \theta}^0={\rm Cent}_{\theta}(\delta,G)^\circ$ is a torus. 
and strongly $\theta$-regular if ${\rm Cent}_{\theta}(\delta,G)$ is abelian. 
In the latter case, the centralizer of $G_{\delta\rtimes \theta}$ in $G$ is 
a maximal torus stable under ${\rm Int}(\delta)\circ \theta$. 
Here $G_{\delta\rtimes \theta}$ is an $F$-algebraic group defined by the functor 
from $F$-algebras to sets;
$$\underline{G}_{\delta\rtimes \theta}:\{F-algebras\}\lra Sets$$
$$R\mapsto \underline{G}_{\delta\rtimes \theta}(R)=
G_{\delta\rtimes \theta}(R):=
\{g\in G(R)=GL_n(E\otimes_F R)\ |\ g^{-1}\delta \theta(g)=\delta\}$$ 
where the action  of $\theta$ on $g$ is induced by the action of $\theta$ on 
$E\otimes_F R$ given by $\theta(e\otimes r)=\theta(e)\otimes r$. 
\item We say a semisimple element $\gamma\in H(F)=GL_n(F)$ is 
strongly $G$-regular if it is a norm of a strongly $\theta$-regular element of $G(F)$. 
\end{enumerate}
\end{Def}
\begin{remark}\label{diagonal}
Fix an isomorphism $T(\oF)\simeq (T_H(\oF))^d$ as before. 
If $\diag(s_1,\ldots,s_n)\in T_H(F)\subset GL_n(F)$ is a norm of 
$\delta=\diag(t_1,\ldots,t_n)\in T(F)=T_H(E)\subset GL_n(E)$, then as a multi-set  
$$\{N_{E/F}(t_1),\ldots,N_{E/F}(t_n)\}=\{s_1,\ldots,s_n\}$$
where $N_{E/F}:E\lra F$ is the usual norm map. 
\end{remark}

\section{The twisted space and the twisted conjugation map }\label{tc}
In this section we follow the notation in Section 3 and 4 of \cite{GV}. 
Let $p$ be the residual characteristic of $F$. 
For any reductive group $\mathcal{G}$ over $F$, we say 
an element of $g$ in $\mathcal{G}(F)$ is topologically unipotent if 
$\ds\lim_{n\to\infty}g^{p^n}=1$. Similarly, for an element 
$X$ in the Lie algebra  ${\rm Lie}(\mathcal{G}(F))$ of $\mathcal{G}(F)$, 
we say $X$ is topologically nilpotent if  $\ds\lim_{n\to\infty}X^{p^n}=0$. 
We denote by $\mathcal{G}(F)_{{\rm tu}}$ (resp. 
${\rm Lie}(\mathcal{G}(F))_{{\rm tn}}$) the set of all topologically unipotent (resp. 
nilpotent) 
elements in $\mathcal{G}(F)$ (resp. ${\rm Lie}(\mathcal{G}(F))$).  

Let us return to our setting. 
Define the twisted space $\tG=G\rtimes \theta$ which is an $F$-variety. 
We have $G\stackrel{\sim}{\lra}\tG,\ g\mapsto g\rtimes \theta$ as an $F$-variety. 
We also define the twisted conjugation map:
\begin{equation}\label{twcm}
{\rm tc}:G(F)\times H(F)\lra \tG(F),\ (g,h)\mapsto g(h\rtimes \theta)
g^{-1}=(gh\theta(g)^{-1})\rtimes\theta
\end{equation}
and put $\mathcal{U}:={\rm tc}(G(F)\times H(F)_{{\rm tu}})$.    
\begin{lem}\label{bijection}Keep the notation being as above. 
Assume the residual characteristic $p$ of $F$ is prime to $d:=[E:F]$. 
The map $H(F)_{{\rm tu}}\lra \mathcal{U},\ h\mapsto h\rtimes \theta$ induces 
a bijection from 
\begin{itemize}
\item the set of $H(F)$-conjugacy classes in  $H(F)_{{\rm tu}}$ to
\item the set of $G(F)$-conjugacy classes in  $\mathcal{U}$.
\end{itemize}
\end{lem}
\begin{proof} We mimic the proof of Lemma 4.0.1, p.996 of \cite{GV}. 
The surjectivity follows from the definition of $\mathcal{U}$. 
Therefore we prove the injectivity. Take two elements $h_1$ and $h_2$ in 
$H(F)_{{\rm tu}}$ such that $h_1\rtimes \theta$ is $G(F)$-conjugate to 
$h_2\rtimes \theta$. Hence, there exists $g\in G(F)$ such that 
$h_2\rtimes \theta=g^{-1}(h_1\rtimes \theta)g=g^{-1}h_1\theta(g)\rtimes \theta$. 
Therefore, $h_2=g^{-1}h_1\theta(g)$. 
Since $h_1$ and $h_2$ are $\theta$-fixed, we have 
$g^{-1}h^2_1\theta^2(g)=h^2_2$. By repeating this, we have 
$$h^d_2=g^{-1}h^d_1\theta^d(g)=g^{-1}h^d_1 g=(g^{-1}h_1g)^d$$ since 
$\theta^d=1$. Pick $a\in \Z_{>0}$ such that $ap\equiv -1$ mod $d$ 
so that for any odd positive integer $n$, we have $(ap)^n+1\equiv 0$ mod $d$. 
It follows from this that $(g^{-1}h_1g)^{(ap)^n+1}=h^{(ap)^n+1}_2$. 
By taking the limit $\ds\lim_{n\to \infty\atop n:{\rm odd}}$ on both sides, 
we have $g^{-1}h_1g=h_2$. For $H=GL_n/F$, stable conjugacy is the same as conjugacy 
and the claim follows.  
\end{proof}
\section{The normalizing factors of orbital integrals}\label{factors}
In this section we collect some facts on the normalizing factors of 
orbital integrals which will be used to define the normalized orbital integrals. 
Let us fix a $p$-adic norm $|\cdot|_F$ on $F$. 
We denote by $\frak g$ (resp, $\frak h$) the Lie algebra of $G(F)$ (resp. $H(F)$). 
Similarly for 
$\delta=g\rtimes \theta$ with an element $g$ in $G(F)$ , 
we define the Lie algebra $\frak g_{\delta}$ for $G_\delta(F)$ where $G_\delta$ is the stabilizer 
of $\delta$ in $G$ which is also given by the $\theta$-twisted stabilizer of $g$ in $G$.   
For $x\in G(F)$ we see 
${\rm Int}(\delta)(x):=\delta x\delta^{-1}={\rm Int}(g)\circ \theta(x)$. Hence 
${\rm Int}(\delta)={\rm Int}(g)\circ \theta$ in ${\rm Aut}(G)$. It also yields 
${\rm Ad}(\delta)={\rm Ad}(g)\circ d\theta$ in ${\rm Aut}(\frak g)$ where 
$d\theta:\frak g\lra \frak g$ is the derivative of $\theta$ at the origin.  
Here ${\rm Ad}(\delta)$ stands for the adjoint action of $\delta$ on $\frak g$ and 
${\rm Ad}(g)$ as well. 
The map 
$d\theta$ is nothing but the natural (Galois) action of $\theta$ on $\frak g=M_n(E)$.  

For $\delta=g\rtimes \theta$ with a semisimple element $g$ in $G(F)$ 
we define the normalizing factor at $\delta$ as follows: 
\begin{equation}\label{nf1}
D_{\tG}(\delta):=|\det(({\rm Ad}(\delta)-1)|\frak g/\frak g_\delta)|_F
\end{equation}
Similarly, for an element $\gamma$ in $H(F)$ 
we define 
\begin{equation}\label{nf2}
D_{H}(\gamma):=|\det((1-{\rm Ad}(\gamma))|\frak h/\frak h_\gamma)|_F
=|\det(({\rm Ad}(\gamma)-1)|\frak h/\frak h_\gamma)|_F
\end{equation}
where $\frak h_\gamma$ is the Lie algebra of $H_\gamma(F)$. 
\begin{lem}\label{ctf}
Assume the residual characteristic $p$ of $F$ is prime to $d:=[E:F]$. 
 If $\gamma\in H(F)\subset G(F)$ is topologically unipotent, then 
$$D_{\tG}(\gamma\rtimes \theta)=D_{H}(\gamma)$$
\end{lem}
\begin{proof}
Clearly, $\frak h=\frak g^{d\theta=+1}=\frak g^\theta$. 
We decompose $\frak g=\frak h\oplus \frak g_1$ where $\frak g_1$ is 
a $d\theta$-stable subspace of $\frak g$ such that any eigenvalue of 
$d\theta|_{\frak g_1}$ is different from $1$. 
Therefore, the all eigenvalues of ${\rm Ad}(\gamma)\circ d\theta-1$ on $\frak g_1$ 
belong to $\mathcal{O}^\times_{\oF}$ 
since  $\gamma$ is topologically unipotent and 
the residual characteristic $p$ of $F$ is prime to $d$. 
It is also clear that $d\theta-1=0$ on $\frak h$ and all eigenvalues of 
${\rm Ad}(\gamma)$ on $\frak g$ belong to 
$\mathcal{O}^\times_{\oF}$ since $\gamma$ is topologically unipotent.   
Note that ${\rm Ad}(\gamma\rtimes \theta)-1={\rm Ad}(\gamma)\circ d\theta-1=
{\rm Ad}(\gamma)(d\theta-{\rm Ad}(\gamma^{-1}))$ and also 
$\frak g_{\gamma\rtimes \theta}=\frak h_\gamma$ 
(since $\frak g^\theta=\frak h$) so that $\frak g/\frak h_\gamma=
\frak h/\frak h_\gamma\oplus \frak g_1$.  
Then we have 
\begin{eqnarray}
D_{\tG}(\gamma\rtimes \theta)&=&
|\det({\rm Ad}(\gamma)|\frak g/\frak h_\gamma)|_F\times 
|\det(d\theta-{\rm Ad}(\gamma^{-1})|\frak g/\frak h_\gamma)|_F 
\nonumber \\
&=& |\det(d\theta-{\rm Ad}(\gamma^{-1})|\frak g/\frak h_\gamma)|_F 
\nonumber \\
&=&|\det(d\theta-1+1-{\rm Ad}(\gamma^{-1})|\frak h/\frak h_\gamma)|_F 
\times |\det(d\theta-{\rm Ad}(\gamma^{-1})|\frak g_1)|_F \nonumber\\
&=&|\det(1-{\rm Ad}(\gamma^{-1})|\frak h/\frak h_\gamma)|_F
=|\det({\rm Ad}(\gamma)-1|\frak h/\frak h_\gamma)|_F\nonumber \\
&=&D_H(\gamma). \nonumber
\end{eqnarray}
\end{proof}

We define the Cayley transform 
\begin{equation}\label{cayley}
\frak c:\frak g\lra G(F),\ X\mapsto \Big(1+\frac{X}{2}\Big)\Big(1-\frac{X}{2}\Big)^{-1}
\end{equation}
which is birationally defined (hence, not defined on the whole space).
By Lemma 3.2.3 of \cite{GV}, $\frak c$ induces 
\begin{equation}\label{homeo}
\frak g_{{\rm tn}}\stackrel{\sim}{\lra} G(F)_{{\rm tu}},\ 
\frak h_{{\rm tn}}\stackrel{\sim}{\lra} H(F)_{{\rm tu}}
\end{equation}
and it satisfies nice properties  
\begin{equation}\label{nice}
\frak c(-X)=\frak c(X)^{-1}\ 
{\rm and}\ {\rm Int}(J)\circ \frak c=\frak c\circ {\rm Ad}(J)
\end{equation}
for any  $X\in \frak g_{{\rm tn}}$ and $J\in G(F)$  (cf. Remark 3.2.4 of \cite{GV}). 
It also holds that $\frak c \circ d\theta=\theta \circ \frak c$. 


Finally, we end this section with the following lemma.
\begin{lem}\label{cprime}The residual characteristic $p$ of $F$ is odd and 
prime to $d=[E:F]$. Then,  
the map $H(F)_{{\rm tu}}\lra H(F)_{{\rm tu}},\ \gamma'\mapsto \gamma'^d$ is a homeomorphism. 
\end{lem}
\begin{proof}It follows from the argument in the proof of Lemma 3.2.7 of \cite{GV}. 
\end{proof}

\section{Normalized orbital integrals and semi-simple descent}\label{oi}
For any topological space $X$ we denote by $C^\infty_c(X)$ the space of 
locally constant functions on $X$ whose supports are compact. 
Recall our twisted space $\tG$. 

For $\delta=\delta_1\rtimes \theta \in \tG(F)$ with a semisimple element $\delta_1\in G(F)$ 
and $\widetilde{f}\in C^\infty_c(\tG(F))$, we define the normalized 
orbital integral of $\widetilde{f}$ at $\delta$ as follows: 
\begin{equation}\label{noiG1}
I_\delta(\wf,dg,dg_{\delta})=D_{\tG}(\delta)^{\frac{1}{2}}\int_{G_\delta(F)\bs G(F)}
\wf(g^{-1}\delta g) dg/dg_\delta
\end{equation}
with respect to the quotient measure $dg/dg_\delta$ of 
measures $dg$ on $G(F)$ and $dg_\delta$  on $G_\delta(F)$. 
The normalizing factor plays an important role when we view the above integral 
 $I_\ast(\wf,dg,dg_{\delta})$ as the Schwartz space of $G(F)$ but we do not 
use this point.   

Let $f$ be the element in $C^\infty_c(G(F))$ corresponding to $\wf$ under the 
isomorphism $G(F)\stackrel{\sim}{\lra}\tG(F),\ g\mapsto g\rtimes \theta$. 
Then, we see that 
\begin{equation}\label{noiG2}
I_\delta(\wf,dg,dg_{\delta})=D_{\tG}(\delta)^{\frac{1}{2}}\int_{G_\delta(F)\bs G(F)}
f(g^{-1}\delta_1 \theta(g)) dg/dg_\delta.
\end{equation}
Similarly, for a semisimple element $\gamma \in H(F)(\subset G(F))$ and  $f^H\in C^\infty_c(H(F))$, we define 
\begin{equation}\label{noiH}
I_\gamma (f^H,dh,dh_\gamma ):=
D_{H}(\gamma )^{\frac{1}{2}}\int_{H_\gamma (F)\bs H(F)}
f^H(h^{-1}\gamma  h) dh/dh_\gamma .
\end{equation}
with respect to the quotient measure $dh/dh_\gamma$ of 
measures $dh$ on $G(F)$ and $dh_\gamma$  on $H_\gamma(F)$. 

Recall $\mathcal{U}={\rm tc}(G(F)\times H(F)_{{\rm tu}})$. Since $\mathcal U$ 
is open in $\tG(F)$ (cf. Lemma 4.0.6 of \cite{GV}), any element in $ C^\infty_c(\mathcal U)$ can be regarded as  
an element in $ C^\infty_c(\tG(F))$. Hence we may assume 
$ C^\infty_c(\mathcal U)\subset  C^\infty_c(\tG(F))$.
\begin{Def}$($Semisimple descent$)$Let $\wf\in C^\infty_c(\mathcal U)\subset  C^\infty_c(\tG(F))$. 
We say an element $f^H$ in $C^\infty_c(H(F)_{{\rm tu}})$ can be obtained 
from $\wf$ by semi-simple descent at $\theta$ with respect to $dg$ and $dh$ if 
for all semi-simple element $\gamma $ in $H(F)_{{\rm tu}}$,
\begin{equation}\label{SSD}
I_\gamma (f^H,dh,dh_\gamma )=I_{\gamma \rtimes\theta}(\wf,dg,dg_{\gamma \rtimes \theta}).
\end{equation}
Notice that $G_{\gamma \rtimes \theta}=H_\gamma $. Therefore, we choose the same measure 
on $G_{\gamma \rtimes\theta}$ and $H_\gamma$ so that $dg_{\gamma \rtimes\theta}=dh_\gamma$. 
Therefore, the semisimple descent depends on the choice of the measures $dg$ and 
$dh$. 
\end{Def}
Next we recall about the transfer matching orbital integrals 
(see (5.5.1), p.71 of \cite{KS}). 
As explained below, we fix a $\theta$-stable Whittaker datum defined in 
Section 5.3 of \cite{KS} for our 
$\theta$-splitting datum ${\rm spl}_G$ introduced at the beginning of Section 
\ref{tes}. This is necessary to normalize transfer factors in Section 5.3 of \cite{KS}. 
\begin{Def}\label{moi}$($Matching of orbital integrals and transfer of Hecke elements$)$ 
We say $\wf\in  C^\infty_c(\tG(F))$ and $f^H \in  C^\infty_c(H(F))$ have matching 
orbital integrals if, for every strongly $G$-regular element $\gamma  \in H(F)\subset G(F)$, 
\begin{equation}\label{transfer}
I_{\gamma }(f^H,dh,dh_\gamma )=\sum_{\delta}\Delta^{{\rm IV}}(\gamma,\delta)
I_{\delta\rtimes \theta}(\wf,dg,dg_{\delta\rtimes\theta})
\end{equation}
where the sum is taken over the set of complete representatives of 
$\theta$-conjugacy classes of strongly $G$-regular elements $\delta$ in $G(F)$ such that 
$\gamma$ is a norm of $\delta$ and the transfer factor $\Delta^{{\rm IV}}(\gamma,\delta)$ 
is defined in Section 5 of \cite{KS} with the collection of \cite{KSc}. 
The matching $($\ref{transfer}$)$ depends on the choice of the measures 
$dg, dh, dg_{\delta\rtimes\theta}$, and $dh_\gamma$ 
(see Section 3.10 of \cite{Wal-book} for measures in the matching of orbital 
integrals). 
\end{Def}
\begin{remark}
Here is to understand the notation in Chapter 5 of \cite{KS} in our setting. 
The symbol $H_1$ there means a $z$-extension of $H=GL_n/F$ and 
obviously $H_1=H=GL_n/F$. Since $H=GL_n/F$, as mentioned before, for elements in $H(F)=GL_n(F)$, 
the notation of stable conjugacy classes is the same as one of conjugacy classes.    
\end{remark}

\section{Transfer factors for our endoscopic datum}\label{tf}
Let us keep the notation in Definition \ref{moi}. Then we will check the following:
\begin{thm}\label{tf-triv}It holds that 
$$\Delta^{{\rm IV}}(\gamma,\delta)=1.$$
\end{thm}
\begin{proof}In what follows, we drop $(\gamma,\delta)$ from the transfer factors. 
Let us recall our twisted endoscopic datum  $(H,\mathcal{H},s,\xi)$ for 
$(G,\theta,\textbf{1}_{{\rm triv}})$  in Section \ref{tes}. Note that 
$\widehat{H}=(\widehat{G}^{\widehat{\theta}})^0$ in the notation of \cite{KS}. 
Because $s=1$, the factor $\Delta_{{\rm I}}$, defined page 33 at line 11 of \cite{KS},  
is trivial. The factor 
$\Delta_{{\rm II}}$ is trivial. 
The factor $\Delta_{{\rm II}}$ in Lemma 4.3.A in page 36 of \cite{KS} 
is rewrote in Section 7.3, p.88 of \cite{Wal-book} and the description there is much easier to 
understand $\Delta_{{\rm II}}$ and in fact, the factor $\Delta_{{\rm II}}=1$ because 
$s = 1$. The factor $\Delta_{{\rm III}}$ defined page 38 of \cite{KS} is trivial because $\textbf{A} = 1$ since $s=1$. The term $\textbf{A} $ measures the difference between 
$\mathcal{H}$ and $\widehat{G}^{\widehat{\theta},0}\rtimes W_F$ . 
There is no difference in our setting because, indeed,  $\mathcal{H}=
\widehat{G}^{\widehat{\theta},0}\rtimes W_F$. The transfer factor $\Delta^{{\rm IV}}$ 
does not contain the terms $\Delta_{{\rm IV}}$. But in fact, even this term is 1 
because its definition in Lemma 4.5.A, page 46 of \cite{KS}, it is as
$\Delta_{{\rm II}}$ a product over an empty set. 
Summing up, we have the claim. 
\end{proof}
\begin{remark}\label{correction}
As Waldspurger pointed out, 
there are some mistakes in \cite{Wal-book} because he has used the
definition of Kottwitz-Shelstad in \cite{KS} but there is a mistake in this definition, cf. Kottwitz-Shelstad ArXiv 2012 \cite{KSc}. However, as he mentioned, this does not affect really 
the proofs in \cite{Wal-book}. 
\end{remark}

\section{Semisimple descent theorem}\label{sdt}
In this section we will prove the semisimple descent theorem for 
some Hecke elements. 
Recall the Lie algebra $\frak g=M_n(E)$ of $G(F)$ and the derivative 
$d\theta$ of $\theta$ at the origin acts there 
by the usual Galois action of $\theta$. We view it as an $\mathcal{O}_F$-module.  
Recall  the open subspace $\mathcal{U}\subset \tG(F)$ defined by 
the twisted conjugation map (\ref{twcm}). The following theorem is an analogue of 
Lemma 4.2.4 of \cite{GV}. 
\begin{thm}\label{mainSSD}
Assume the residual characteristic $p$ of $F$ satisfies $p>2$ and $p\nmid d=[E:F]$. 
 Let $L\subset \frak g_{{\rm tn}}$ be a $d\theta$-invariant 
 $\mathcal{O}_F$-lattice satisfying $L\cdot L\subset \varpi_F L$. 
Write $L=L_1\oplus L_2$ where $L_1$ is an $\O_F$-lattice of $\frak g^\theta$
and $L_2$  is an $\O_F$-lattice of $(1-d\theta)\frak g$ 
$($this is always possible by assumption on the residual characteristic$)$. 
Let $K_L:=\frak c(L)$ and $K_{L,\theta}=\frak c(L)\cap H(F)$ 
so that $K_{L,\theta}=K^\theta_L$ by $\frak c\circ d\theta=\theta\circ \frak c$. 
Then it holds that 
\begin{enumerate}
\item $K_L\subset G(F)_{{\rm tu}}$ $($resp. $K_{L,\theta}\subset H(F)_{{\rm tu}})$ 
is a compact open subgroup of $G(F)$ $($resp. $H(F))$;
\item $K_L\rtimes \theta={\rm tc}(K_L,K_{L,\theta})$;
\item Let $C_H\subset H(F)_{{\rm tu}}$ be an open compact subset and 
$K\subset G(F)$ be an open compact subgroup. 
Assume $C_H$ is invariant under conjugation by $K\cap H(F)$.  
Put $C={\rm tc}(K,C_H)$ which is 
an open compact subset of $\mathcal{U}\subset \tG(F)$. 
Then the element ${\rm vol}^{-1}(K\cap H(F))\textbf{1}_{C_H}$ of 
$C^\infty_c(H(F))$ can be obtained from the element 
${\rm vol}^{-1}(K)\textbf{1}_{C}$ of $C^\infty_c(\mathcal{U})\subset 
C^\infty_c(\tG(F))$ by semisimple descent at $\theta$. 
In particular, ${\rm vol}^{-1}(K_{L,\theta})\textbf{1}_{K_{L,\theta}}$ 
can be obtained from the element 
${\rm vol}^{-1}(K_L)\textbf{1}_{K_L\rtimes \theta}$ by semisimple descent at $\theta$.
\end{enumerate}
\end{thm}
\begin{proof}
 Let us follow the proof of Lemma 4.2.4 of 
\cite{GV} verbatim with a minor modification.   
The first claim follows from Lemma 4.2.3, p.998 of \cite{GV}. 
Next let us prove the second claim. 
Obviously, ${\rm tc}(K_L,K_{L,\theta})$ is contained in $K_L\rtimes \theta$. 
As explained in the proof of Lemma 4.2.4-(i) of \cite{GV}, ${\rm tc}$ is an 
open map. Since $L_1$ is compact, there exists $m_0\in \Z_{\ge 0}$ such that 
${\rm tc}(K_L,K_{L,\theta})\supset \frak c(L_1+\varpi^{m_0}_F L_2)\rtimes \theta$. 
By reverse induction, for any $m\in \Z_{\ge 0}$, each element of  
$\frak c(L_1+\varpi^{m}_F L_2)$ can be written as 
$g^{-1}\delta \theta(g)$ for some $g\in K_L$ and $\delta\in 
\frak c(L_1+\varpi^{m+1}_F L_2)$. 
Since $\frak c(L)=1+L$ (cf. Lemma 4.2.3 of \cite{GV}), any element $\frak c(L_1+\varpi^{m}_F L_2)$ is written by 
$1+X$ for some  $X=X_1+X_2$ with $X_1\in L_1$ and $X_2\in \varpi^m_F L_2$. 
By assumption on $p$, $F(\zeta_d)/F$ is unramified. Therefore we may 
suppose that $F$ contains a primitive $d$-th root of unity in the argument below.  
Let us write $X_2=\ds\sum_{i}X_{2,i}$ where $\{X_{2,i}\}_i$ are eigenvectors of $d\theta$ 
with $d\theta X_{2,i}=\zeta^{n_i}_d X_{2,i},\ \zeta^{n_i}_d\neq 1$. 
Since $d\theta$ has the distinct eigenvalues, the $\zeta^{n_i}_d$'s are 
different each other. Further, $1-\zeta^{n_i}_d$ is a unit in $\O_F$ by assumption on 
$p$ and $d$. 
Thus, each $X_{2,i}$ belongs to $\varpi^m_F L_2$. 
Put $Y=\sum_{i}a_i X_{2,i}\in \varpi^m_F L_2$ with $a_i=(1-\zeta^{n_i}_d)^{-1}$. 
As explained in the mid of the proof of Lemma 4.2.4 of \cite{GV}, 
$\frak c(Y)\equiv 1+Y\ {\rm mod}\ \varpi^{m+1}_F$. Thus, using 
the assumption $L\cdot L\subset \varpi_F L$,  
\begin{eqnarray}
\frak c (Y)^{-1}(1+X)\theta(\frak c(Y))&=&
\frak c (Y)^{-1}(1+X)\frak c(d\theta(Y)) \nonumber \\
&\equiv &  
(1-Y)(1+X)(1+\sum_{i}\zeta^{n_i}_d a_i X_{2,i})
\ {\rm mod}\ \varpi^{m+1}_F \nonumber  \\
&\equiv &  
1-Y+X_1+X_2+\sum_{i}\zeta^{n_i}_d a_i X_{2,i}
\ {\rm mod}\ \varpi^{m+1}_F \nonumber \\
&=& 1+X_1  \ {\rm mod}\ \varpi^{m+1}_F. \nonumber 
\end{eqnarray}
By using $\frak c(L)=1+L$ again, we have the claim. 

Finally let us prove the third claim, namely, the semisimple descent. 
Let $\gamma\in H(F)_{{\rm tu}}$. 
By Lemma \ref{ctf}, $D_{\tG}(\gamma\rtimes\delta)=D_H(\gamma)$. 
If the characteristic polynomial of $\gamma$ has multiple roots, then 
$D_H(\gamma)=0$. Therefore, we may assume that the eigenvalues of 
$\gamma$  are different each other. In particular, $G_{m\rtimes \delta}=H_\gamma$ is a 
torus. In particular $H_\gamma(F)$ is unimodular and so are $G(F)$ and $H(F)$.  
What we need to check is, by definition, the equality without 
normalizing factors:
$${\rm vol}^{-1}(K\cap H(F))\int_{H_\gamma(F)\bs H(F)}
\textbf{1}_{C_H}(h^{-1}\gamma h) dh/dh_\gamma$$
$$={\rm vol}^{-1}(K)\int_{G_{\gamma\rtimes\theta}(F)\bs G(F)}
\textbf{1}_{C}(g^{-1}(\gamma\rtimes \theta) g) dg/dg_{\gamma\rtimes \theta}.$$
However, 
applying Lemma 4.2.5 of \cite{GV} to $G_1=G,\ H_1=H,\ I_1=H_\gamma$, we have 
the claim. 
\end{proof}

\section{Proofs of the main theorems}\label{proof}
We first prove Theorem \ref{main1} and Theorem \ref{main2} 
simultaneously, and  then Theorem \ref{main3} after that.   
\begin{proof}(Proofs of Theorem \ref{main1} and Theorem \ref{main2}) 
Recall $f^H={\rm vol}(K_E(m)_\theta)\textbf{1}_{K_E(m)_\theta}$ and 
$f^G={\rm vol}(K_E(m))\textbf{1}_{K_E(m)}$. Put  
$\wf^G={\rm vol}(K_E(m))\textbf{1}_{K_E(m)\rtimes \theta}$ corresponding to $f^G$ under 
$G(F)\stackrel{\sim}{\lra} \tG(F),\ g\mapsto g\rtimes\theta$. 
In view of matching, if a semisimple element $\gamma$ in $H(F)$ is not regular, 
then $\delta\in G(F)$ is not 
$\theta$-regular for 
which $\gamma$ is a norm of $\delta$. In fact, it follows from the proof of 
Lemma 1.1 of \cite{AC}, $G_{\delta\rtimes \theta}$ is an inner form of $H_\gamma$. 
The claim follows from this. Therefore, the both sides in Theorem 
\ref{main2} vanish because of the normalizing factors which vanish on such elements.  
Thus we may assume that $\gamma\in H(F)$ is regular, then $H_\gamma$ is a torus and it is isomorphic to $G_{\delta\rtimes \theta}={\rm Cent}_\theta(\delta,G)$. In particular, it is abelian. 
Therefore, $\delta$ is a strongly $\theta$-regular element.  The sum of the right hand side 
in (\ref{transfer}) runs over a set of complete representatives of 
$G(F)$-conjugacy classes of regular elements $\delta$ such that $\gamma$ is 
a norm of $\delta$. Plugging this into Theorem \ref{tf-triv}, 
\begin{equation}\label{eq1}
\sum_{\delta}\Delta^{{\rm IV}}(\gamma,\delta)
I_{\delta\rtimes \theta}(\wf^G,dg,dg_{\gamma\rtimes\theta})=
\sum_{\delta:{\rm regular} \atop 
\text{$\gamma$ is a norm of $\delta$}}I_{\delta\rtimes \theta}(\wf^G,dg,dg_{\delta\rtimes\theta}).
\end{equation}
Suppose that $\gamma\not\in H(F)_{{\rm tu}}$. For any $\delta$ on the right in the above sum, 
it holds that $\delta\not\in G(F)_{{\rm tu}}$ by Remark \ref{diagonal}.   
 Since $f^H$ and $\wf^G$ are 
supported on topologically unipotent elements, the equality in the claim trivially holds. 
Thus, we may assume that $\gamma\in H(F)_{{\rm tu}}$ and  
$\delta\in G(F)_{{\rm tu}}$ for $\gamma$ and $\delta$ in question.  
Let $\delta_1,\delta_2$ be 
elements in $G(F)_{{\rm tu}}$ such that $\gamma$ is a norm of both of them. 
By Lemma 1.1-(ii) of \cite{AC}, $\delta_1$ is $\theta$-conjugate to $\delta_2$. 
Therefore,  (\ref{eq1}) becomes 
\begin{equation}\label{eq2}
\sum_{\delta}\Delta^{{\rm IV}}(\gamma,\delta)
I_{\delta\rtimes \theta}(\wf^G,dg,dg_{\delta\rtimes\theta})=
I_{\delta\rtimes \theta}(\wf^G,dg,dg_{\delta\rtimes\theta})
\end{equation}
for any such $\delta$. 
By Lemma \ref{cprime}, there exists an element $\gamma'\in H(F)_{{\rm tu}}$ 
such that $\gamma'^d=\gamma$. 
Thus $N_{E/F}(\gamma')=\gamma'^d=\gamma$. 
By Lemma 1.1-(ii) of \cite{AC} and Lemma \ref{bijection}, 
we may assume that $\delta=\gamma'$. 
To be more precise, 
\begin{equation}\label{eq3}
I_{\delta\rtimes \theta}(\wf^G,dg,dg_{\delta\rtimes\theta})
=I_{\gamma'\rtimes \theta}(\wf^G,dg,dg_{\gamma'\rtimes\theta})
\end{equation} for any 
element $\gamma'$ of $H(F)$ which is $H(F)$-conjugate to an element whose usual norm is $\gamma$.  
Since $f^G$ is supported on $K_E(m)$ and the homeomorphism in 
Lemma \ref{cprime} induces a homeomorphism $K_E(m)\stackrel{\sim}{\lra}K_E(m)$. 
Thus, 
\begin{equation}\label{eq4}
I_{\gamma'\rtimes \theta}(\wf^G,dg,dg_{\gamma'\rtimes\theta})=
I_{\gamma\rtimes \theta}(\wf^G,dg,dg_{\gamma\rtimes\theta}).
\end{equation}
Further, it follows from Theorem \ref{mainSSD}-(3) that 
\begin{equation}\label{eq5}
I_{\gamma\rtimes \theta}(\wf^G,dg,dg_{\gamma\rtimes\theta})=
I_{\gamma}(f^H,dh,dh_{\gamma}).
\end{equation}
Theorem \ref{main1} follows from the equations (\ref{eq1}) through (\ref{eq5}). 
Theorem \ref{main2} follows from the equations (\ref{eq3}) through (\ref{eq5}). 
\end{proof}

\begin{proof}
(A proof of Theorem \ref{main3}) 
As the referee suggested, we shall prove only a local version and the global one is an immediate consequence.  
Let us keep the notation in Theorem \ref{main1}. 
Let $\pi$ be an irreducible admissible representation of $H(F):=GL_n(F)$.  
Suppose we have a local base change $\Pi$ of $\pi$ which is an irreducible admissible representation of $G(F)=GL_n(E)$ 
(see \cite[the remark at p.59 from line -9]{AC} for the existence). Let $V$ be the representation space of $\Pi$ and 
$\Pi^\theta$ be the twist of $\Pi$ by $\theta$ which is given by the action of 
$\Pi^\theta(g)=\Pi(\theta(g))$ on $V$ for each $g\in G(F)$. 
Let $I_\theta:V\lra V$ be the intertwining operator defined in \cite[Chapter 1, Section 2, p.10]{AC} satisfying $I^d_\theta={\rm id}_V$ and $\Pi I_\theta=I_\theta \Pi^\theta$. 
Recall $${\rm tr}(\pi(f)):=\int_{H(F)} {\rm ch}_{\pi}(g) f(g) dg,\ f\in C^\infty_c(H(F))$$ 
where ${\rm ch}_{\pi}$ is the character of $\pi$, and the integral on the right hand side of the equation can be decomposed into the orbital integrals and the integral over the orbital integrals (cf .\cite[Theorem 5.11]{HC}). A similar integral is obtained for twisted traces (cf. \cite[p.151,Example, p.153,Theorem 2]{Clozel}). 
Since $f^H,f^G$ are the functions with matching orbital integrals, 
Definition 6.1 in \cite{AC} and Definition \ref{moi} 
implies 
\begin{equation}\label{twisted-match}
{\rm tr}(\pi(f^H))={\rm tr}(\Pi(f^G)I_\theta). 
\end{equation}
It follows from this and $\Pi I_\theta=I_\theta \Pi^\theta$ that 
$${\rm dim}_\C\pi^{K_E(m)_\theta}={\rm tr}(\pi(f^H))={\rm tr}(I_\theta|_{V^{K_E(m)}})$$
(we remark that $K_E(m)$ is $\theta$-stable and $V^{K_E(m)}$ is $I_\theta$-stable).
Put $N:=\dim_\C V^{K_E(m)}$. Since $I_\theta^d={\rm id}_V$, $I_\theta|_{V^{K_E(m)}}$ 
is diagonalizable and therefore we have a basis $\{v_i\}_{i=1}^N$ such that 
for each $1\le i\le N$, $I_\theta(v_i)=
\zeta^{n_i}_d v_i$ for some $n_i\in\Z$ where $\zeta_d$ stands for a primitive $d$-th 
root of unity.
Then we have  
$${\rm dim}_\C\pi^{K_E(m)_\theta}= |{\rm tr}(I_\theta|_{V^{K_E(m)}})|
=\Big|\sum_{i=1}^N\zeta^{n_i}_d\Big|\le N=\dim_\C V^{K_E(m)}.$$
This completes the proof. 
\end{proof}


\begin{thebibliography}{99}

\bibitem{AC}J. Arthur and L. Clozel, Simple algebras, base change, 
and the advanced theory of the trace formula.
Annals of Mathematics Studies, 120. Princeton University Press, Princeton, NJ, 1989. xiv+230.

\bibitem{Clozel}L. Clozel, Characters of nonconnected, reductive $p$-adic groups. Canad. J. Math. 39 (1987), no. 1, 149–167.

\bibitem{CR}L. Clozel and C. Rajan, Solvable base change. 
J. Reine Angew. Math. 772 (2021), 147–174. 11 (22). 

\bibitem{Fe}A. Ferrari,  
Th\`eor\'me de l'indice et formule des traces. 
Manuscripta Math. 124 (2007), no. 3, 363–390.

\bibitem{GV}R. Ganapathy and S. Varma, {\em On the local Langlands correspondence for split classical groups over local function fields},  J. Inst. Math. Jussieu {\bf 16} (2017), no. 5, 987-1074.

\bibitem{HC}Harish-Chandra, Admissible invariant distributions on reductive p-adic groups. With a preface and notes by Stephen DeBacker and Paul J. Sally, Jr. University Lecture Series, 16. American Mathematical Society, Providence, RI, 1999. xiv+97 pp. 

\bibitem{HI}K.  Hiraga and A. Ichino, On the Kottwitz-Shelstad normalization of transfer factors for automorphic induction for $GL_n$. Nagoya Math. J. 208 (2012), 97–144.

\bibitem{J}H. Jacquet, On the base change problem: after J. Arthur and L. Clozel. Number theory, trace formulas and discrete groups (Oslo, 1987), 111–124, Academic Press, Boston, MA, 1989.

\bibitem{KWY}H-H. Kim, S. Wakatsuki, and T. Yamauchi, {\em Equidistribution theorems for holomorphic Siegel cusp forms of general degree}, a preprint 2021. 



\bibitem{Ko}R-E. Kottwitz, Base change for unit elements of Hecke algebras. Compositio Math. 60 (1986), no. 2, 237–250. 

\bibitem{KS}R-E. Kottwitz and D. Shelstad, 
Foundations of twisted endoscopy. Astérisque No. 255 (1999), vi+190 pp. 

\bibitem{KSc}R-E. Kottwitz and D. Shelstad, On Splitting Invariants and Sign Conventions in Endoscopic Transfer, a preprint available on Arxiv referred by arXiv:1201.5658. 

\bibitem{LR} J. Lansky and A. Raghuram, {\em On the correspondence of representations between $GL(n)$ and division algebras}, 
Proc. Amer. Math. Soc, {\bf 131} (2003), 1641--1648.

\bibitem{NgoICM}B-C. Ng\^{o},  
Endoscopy theory of automorphic forms. Proceedings of the International Congress of Mathematicians. Volume I, 210–237, Hindustan Book Agency, New Delhi, 2010.

\bibitem{Oi} M. Oi, {\em Depth preserving property of the local Langlands correspondence for quasi-split classical groups in a large residual characteristic}, 
Manuscripta math. (2022)., https://doi.org/10.1007/s00229-022-01397-9.



\bibitem{Wal-book}J-L. Waldspurger, L'endoscopie tordue n'est pas si tordue.Mem. Amer. Math. Soc. 194 (2008), no. 908, x+261 pp. 

\end{thebibliography}
\end{document}